\documentclass[12pt,reqno]{amsart}
\newcommand{\ben}{\begin{enumerate}}
\newcommand{\een}{\end{enumerate}}
\newcommand{\bqu}{\begin{quote}}
\newcommand{\equ}{\end{quote}}
\newcommand{\beq}{\begin{equation}}
\newcommand{\eeq}{\end{equation}}
\newcommand{\bec}{\begin{center}}
\newcommand{\ece}{\end{center}}

\usepackage{url}
\usepackage{mathtools}
\usepackage{latexsym,epsfig,amssymb,amsmath,amsthm,color,url}
\usepackage[shortlabels]{enumitem}
\usepackage{hyperref}

\allowdisplaybreaks 
 \numberwithin{equation}{section}
\newtheorem{theorem}{Theorem}[section]

\newtheorem{lemma}[theorem]{Lemma}

\newtheorem{example}[theorem]{Example} 
\newtheorem{ind}[theorem]{Induction hypothesis}
\newtheorem{assume}[theorem]{Assumption}
\newtheorem{conc}[theorem]{Conclusion}

\theoremstyle{definition}
\newtheorem{defn}[theorem]{Definition}

\theoremstyle{remark}
\newtheorem{rem}[theorem]{Remark}

\DeclareMathOperator{\kein}{KEIN}
\DeclareMathOperator{\EHR}{EHR}
\DeclareMathOperator{\ehr}{EHR_{alt}}
\DeclareMathOperator{\ehrw}{EHR_{weak}}

\title{Quantifier alternation in a class of recursively defined tree properties}
\date{}
\author{Moumanti Podder}
\address{Moumanti Podder, \ Department of Mathematics, \ University of Washington, \ C-524 Padelford Hall, West Stevens Way Northeast, Seattle, WA 98105, United States.}
\email{mpodder3@uw.edu.}

\begin{document}
\bibliographystyle{plainnat}


\begin{abstract}
Alternating quantifier depth is a natural measure of difficulty required to express first order logical sentences. We define a sequence of first order properties on rooted, locally finite trees in a recursive manner, and provide rigorous arguments for finding the alternating quantifier depth of each property in the sequence, using Ehrenfeucht-Fra\"{i}ss\'{e} games.
\end{abstract}



\thanks{The author was partially supported by the grant NSF DMS-1444084.}

\maketitle

\section{Introduction}
\sloppy We study the alternating quantifier depths of a class of recursively defined first order properties on rooted, locally finite trees. A tree $T$ is a connected graph on either a finite or an infinite vertex set that does not contain any cycle. We in particular shall consider trees as directed graphs, with clear distinction of parent and child between adjacent vertices. A rooted tree has a special vertex called the root, which we denote by $R$. A tree is called locally finite when each vertex of the tree has finite degree. We denote by $V = V(T)$ the vertex set of the tree $T$. For a given rooted tree $T$ and a vertex $v \in V(T)$, we let $T(v)$ denote the subtree of $T$ that consists of $v$ and all its descendants.

We now define the first order (FO) language on rooted trees. This consists of the root $R$ of the tree as a constant symbol, the equality of vertices (denoted $x = y$ which means that the vertices $x$ and $y$ coincide), and the parent-child relation (denoted $\pi(y) = x$ which means that the vertex $x$ is the parent of the vertex $y$). We denote vertices in the tree other than the root by letters such as $x, y, z, \ldots$ usually in lower case. Sentences are finite and comprise Boolean connectives such as $\neg, \wedge, \vee, \implies, \Leftrightarrow$ etc., and existential (denoted $\exists$) and universal (denoted $\forall$) quantification over vertices. We refer the reader to any of \cite{02}, \cite{03} and \cite{04} for detailed references on first order logic. The \emph{quantifier depth}, abbreviated as q.d., of an FO sentence $A$ is the minimum number of nested quantifiers required to express $A$, whereas the \emph{alternating quantifier depth} of $A$, which we abbreviate henceforth as the a.q.d.\ of $A$, is the minimum number, the minimum being taken over all formulas that express $A$, of the maximum number of alternating nested quantifiers required to write each such formula. That is, a.q.d.\ gives us the minimum number of times we have to switch from existential to universal or universal to existential quantification in a nested sequence of quantifiers in order to express the sentence. Trivially, the a.q.d.\ of any FO sentence is always bounded above by its q.d. Purely existential or purely universal FO sentences are defined to have a.q.d.\ $0$.


\begin{example}
Consider the FO sentence: there exists a vertex with precisely one child. This can be written as
\begin{equation}
\exists x \left[\exists y \left[\left[\pi(y) = x\right] \wedge \forall z \left[\pi(z) = x \implies z = y\right]\right]\right].
\end{equation}
The q.d.\ of this sentence is $3$ whereas the a.q.d.\ is $1$, because we have two existential followed by one universal quantifier in the nested sequence.

\end{example}

One of the classically studied family of questions by mathematical logicians comprise the model-theoretic results about the expressive power of logical languages. In general, given two languages $A$ and $B$ such that $A$ is a subset of $B$, the task is to show that the two are not tautologically equal, i.e.\ there exists a sentence which is expressible in $B$ but not in $A$. Quantifier alternation hierarchy is a natural tool for proving such results. We now discuss some literature on the a.q.d.\ of FO, and also touch briefly on a.q.d.\ of monadic second order (MSO) logic. \cite{05} is a survey paper that discusses the usefulness of quantifier alternation hierarchy in membership algorithms -- algorithms that are used to decide whether a given regular language of finite words is definable by a sentence from FO logic or not. They beautifully express the necessity to understand quantifier alternation hierarchy as a measure of the difficulty of defining a language -- a language is considered complicated if many switches need to be made between blocks of existential quantifiers and blocks of universal quantifiers. \cite{06} considers quantifier alternation hierarchy within FO language comprising two variables over finite words with linear order and binary successor predicate. They show that for a given regular language and any non-negative integer $m$, it is decidable whether the language is definable by an FO sentence of a.q.d.\ at most $m$. \cite{07} also considers quantifier alternation hierarchy in FO logic on finite words -- they prove that one can decide membership of a regular language to the levels of FO sentences that have a.q.d.\ $1$ or FO sentences with a.q.d.\ $2$ and beginning with an existential quantifier. \cite{08} provides an effective characterization of tree languages that are definable by FO sentences with a.q.d.\ $1$. \cite{15} proves the strictness of FO quantifier alternation hierarchy over the class of finite labeled graphs  -- for each positive integer $k$, they exhibit a property of finite labeled directed graphs that is expressible as an FO sentence of a.q.d.\ $k+1$ but not as any FO sentence of a.q.d.\ $k$. \cite{16} shows that FO quantifier alternation hierarchy is equivalent to dot-depth alternation hierarchy for FO formulas over word models with a total ordering on the alphabet rather than the successor relation on word positions. \cite{13}, \cite{14} and \cite{17} discuss the strictness of MSO quantifier alternation, i.e.\ sentences in prenex normal form having a prefix of $k+1$ many alternations of set quantifiers can describe strictly more graph properties than those having a prefix of only $k$ many alternations of set quantifiers. 

All the results we have been able to find in the literature pertaining to quantifier alternation hierarchy as an important tool to understand the expressive power of languages, are in the premise of graphs, trees, grids or pictures where each vertex or position is labeled by some element from a finite alphabet. In particular, in case of trees, some form of ordering is considered on the vertices, and each vertex is assigned a label from the pre-fixed, finite alphabet, and the unary predicate that specifies the label of a vertex is considered. We hope that our attempt to investigate a.q.d.'s of FO properties of rooted trees without the involvement of any alphabet will serve to begin a new direction of study with new definitions of FO and MSO languages on rooted, locally finite trees. Studies in this direction may reveal strictness, or lack thereof, of quantifier alternation hierarchy of the FO language, as we define it, on rooted trees. Continuing with this hope, we cite here \cite{12}, where the authors define the FO language on graphs using the binary relations of adjacency and equality of vertices (without the involvement of any alphabet). They show that the minimum number of quantifier alternations that an FO sentence $A$ must have in order to fail to have a $0-1$ law on $G(n, n^{-\alpha})$ for infinitely many values of $\alpha$ is $3$. We also hope that future work in this area would bring to light many more classes of recursively defined properties on rooted trees whose analysis may reveal more information on quantifier alternation hierarchy.


\subsection{Organization of the paper:} Our paper is organized in the following manner. In Section~\ref{sec:description} we describe the class of recursively defined properties on rooted, locally finite trees that we examine. In Section~\ref{sec:Ehrenfeucht} we describe the main tool we use to understand the a.q.d.\ of these properties -- the Ehrenfeucht games, also known as the Ehrenfeucht-Fra\"{i}ss\'{e} games. In Section~\ref{sec:construction} we describe the detailed inductive construction of the trees where the Ehrenfeucht games are played, and finally, in Section~\ref{sec:winning}, we describe, also along an inductive argument, the winning strategy for Duplicator, that gives us the final conclusion.

\section{Description of the problem}\label{sec:description}
We define, for any rooted, locally finite tree $T$ with root $R$, and any $x \in V(T)$, the property $P_{0}(x)$ which states that $x$ has no child, which can be expressed as
\begin{equation}
P_{0}(x) = \forall y \neg\left[\pi(y) = x\right].
\end{equation}
We now define the class of properties $P_{i}(x)$, for $i \in \mathbb{N}$ and any $x \in V(T)$, recursively as follows:
\begin{equation}\label{recur_prop}
P_{i}(x) := \forall \ y \left[\pi(y) = x \implies \neg P_{i-1}(y)\right].
\end{equation}
For example, $P_{1}(x)$ denotes the property that $x$ has no child with no child. In particular, we define the property $\kein_{i} = P_{i}(R)$ for every $i \in \mathbb{N}$. The aim of this paper is to show that the a.q.d.\ of $\kein_{i}$ is $i$ for every $i \in \mathbb{N} \cup \{0\}$. It is immediate to see that $\kein_{0}$ has a.q.d.\ $0$, since it is a purely universal sentence. From our recursive definition \eqref{recur_prop}, we can also see that the a.q.d.\ of $\kein_{i}$ is at least $i$. It remains to be shown that $i$ is indeed the minimum number of alternations of nested quantifiers required to express the sentence $\kein_{i}$ for every $i \in \mathbb{N}$.

We mention here an analogous family of properties defined for graphs. On graphs, the corresponding FO language will consist of vertices denoted by $x, y, z, \ldots$ etc.\, the equality of vertices (again, denoted $x = y$) and the adjacency of vertices (denoted $x \sim y$), the Boolean connectives and the existential and universal quantifiers over vertices. One can define, for any vertex $x$, the property $Q_{0}(x)$ that $x$ has no neighbour, or, in other words, $x$ being an isolated vertex, and set $N_{0} = \exists x [Q_{0}(x)]$. Then one can define
\begin{equation}
Q_{1}(x) = \forall y [y \sim x \implies [\exists z [\neg[z = x] \wedge [z \sim y]]]],
\end{equation}
or, in other words, for every neighbour $y$ of $x$, there exists some neighbour $z$ of $y$ which is distinct from $x$, or, in other words, no neighbour of $x$ is a degree $1$ vertex. We then define $N_{1} = \exists x [Q_{1}(x)]$. Finally, we recursively define
\begin{equation}
Q_{i}(x) = \forall y[y \sim x \implies \neg[Q_{i-1}(y)]],
\end{equation}
and set $N_{i} = \exists x [Q_{i}(x)]$. Notice that although defined in a rather similar fashion to the tree properties $\kein_{s}$, the properties $N_{s}$ are harder to analyze. This is because of the following reason. As we shall see in the subsequent sections, our argument hinges upon the use of Ehrenfeucht games on two trees that we construct recursively on $s$, such that one of them satisfies $\kein_{s}$ and the other satisfies $\neg \kein_{s}$. The attempt to do a similar recursive construction in the graph case fails, because there, the property $\neg N_{s}$ claims something much stronger than $\neg \kein_{s}$ -- it claims that a certain property fails to hold for \emph{every} vertex in the graph. The study of the properties $N_{s}$ for $s \in \mathbb{N}$ and their a.q.d.\ remains of keen interest to us in future research in this area. The comparison of such classes of properties on graphs and rooted trees may reveal significant difference between FO on graphs and FO on rooted trees as we have defined them. 

\section{The Ehrenfeucht game for alternating quantifiers}\label{sec:Ehrenfeucht}
As mentioned above, the main objective of this paper is to show that $\kein_{s}$ has a.q.d.\ precisely equal to $s$ for every $s \in \mathbb{N}$. The rigorous proof of this statement relies on a special version of the well-known Ehrenfeucht-Fra\"{i}ss\'{e} games, which we henceforth refer to as simply the Ehrenfeucht games. For standard definition of these combinatorial games and their connection to mathematical logic, we refer the reader to any one of the references \cite{01} and \cite{03}. For the special version of the game we are about to make use of in this paper, we refer the reader to Definition~2.7, Definition~2.8 and Theorem~2.9 of \cite{10}, and a formal proof of the connection of this special version with a.q.d.\ of FO sentences in \cite{11}. 

We first state here, in the premise of rooted trees, the version of the Ehrenfeucht games described in \cite{10} and \cite{11}, that determines the maximum among the a.q.d.'s of all FO sentences that hold true in both the structures on which the game is being played. We first state here a few general rules and terminology which apply to all three of Definitions~\ref{EHR_first}, \ref{EHR_alt} and \ref{weaker_EHR}. The game is played, for a given number of rounds, on two given rooted trees $T_{1}$ and $T_{2}$, by two players known as Spoiler and Duplicator. Each round of the game consists of two parts: a \emph{move} by Spoiler followed by a move by Duplicator. By a move, we mean the action by any player of choosing a vertex from one of the two trees. In each round, once Spoiler has made his selection of a vertex from one of the trees, Duplicator must make her selection of a vertex from the \emph{other} tree. Thus, in every round, there is precisely one vertex chosen from $T_{1}$ and one from $T_{2}$.

If in the $i$-th round, for some $1 \leq i \leq r-1$ where $r$ is the total number of rounds, Spoiler makes his move on $T_{1}$ and in the $(i+1)$-st round makes his move on $T_{2}$, or vice versa, then we say that a \emph{switch} has happened. 
\begin{defn}\label{EHR_first}[Ehrenfeucht game with given maximum number of alternations]
Given two rooted trees $T_{1}$ with root $R_{1}$ and $T_{2}$ with root $R_{2}$, and two positive integers $r$ and $s$ with $r \geq s$, this game, denoted $\EHR\left[T_{1}, T_{2}, s, r\right]$, consists of $r$ many rounds. Spoiler is allowed to make his move in the first round on \emph{any} of the two trees $T_{1}$ and $T_{2}$, but throughout the game, he is allowed to make at most $s$ many switches. 

Let $x_{i}$ be the vertex selected from $T_{1}$ and $y_{i}$ that from $T_{2}$ in round $i$, for $1 \leq i \leq r$. We set $x_{0} = R_{1}$ and $y_{0} = R_{2}$. Duplicator wins the game if \emph{all} of the following conditions hold: for all $i, j \in [r]$,
\begin{enumerate}[label={(Main \arabic*)},leftmargin=*]
\item \label{main_1} $\pi(x_{j}) = x_{i} \Leftrightarrow \pi(y_{j}) = y_{i}$;
\item \label{main_2} $x_{i} = x_{j} \Leftrightarrow y_{i} = y_{j}$,
\end{enumerate} 
where for any positive integer $n$, we denote by $[n]$ the set $\left\{0,1, \ldots, n\right\}$.
\end{defn}

The importance of this version of the Ehrenfeucht games is illustrated in the following theorem, which is Theorem~2.9 of \cite{10} stated for rooted, locally finite trees.
\begin{theorem}\label{EHR_logic}
For any two rooted trees $T_{1}$ and $T_{2}$, for positive integers $r$ and $s$, Duplicator wins $\EHR\left[T_{1}, T_{2}, s, r\right]$ if and only if for every FO sentence $A$ of q.d.\ at most $r$ and a.q.d.\ at most $s$, we have $T_{1} \models A \Leftrightarrow T_{2} \models A$.
\end{theorem}
Here, for any rooted tree $T$ and any FO sentence $A$, the notation $T \models A$ implies that $A$ holds in $T$. The conclusion of Theorem~\ref{EHR_logic} is that, if Duplicator wins $\EHR\left[T_{1}, T_{2}, s, r\right]$, then for every FO sentence $A$ of q.d.\ at most $r$ and a.q.d.\ at most $s$, either $A$ holds for both $T_{1}$ and $T_{2}$, or it holds for neither.

The version of the Ehrenfeucht games that we use is slightly different from that given in Definition~\ref{EHR_first}. We show in Lemma~\ref{EHR_implies_EHR_alt} that if, on two given trees $T_{1}$ and $T_{2}$, Duplicator wins the Ehrenfeucht game described in Definition~\ref{EHR_alt} with sufficiently large values of the parameters concerned, then she also wins the Ehrenfeucht game described in Definition~\ref{EHR_first}.
\begin{defn}\label{EHR_alt}[Ehrenfeucht game with alternation after every $k$ rounds for fixed $k$]
Given two rooted trees $T_{1}$ with root $R_{1}$ and $T_{2}$ with root $R_{2}$, and two positive integers $k$ and $s$, this game, denoted $\ehr\left[T_{1}, T_{2}, s; k\right]$, consists of $s k$ rounds. The $s k$ many rounds are divided into $s$ many batches of $k$ rounds each. Spoiler, before the very first round, selects any one of $T_{1}$ and $T_{2}$ and in the first $k$ rounds makes his moves on that tree, in the next $k$ rounds he makes his moves on the other tree, and so on, i.e.\ he makes a switch after every $j k$-th round for every $1 \leq j \leq s-1$ (in particular, there can be no switch possible if $s = 1$).  

If $x_{i}$ is the vertex selected from $T_{1}$ and $y_{i}$ that from $T_{2}$ in round $i$, for $i \in [s k]$, setting $x_{0} = R_{1}$ and $y_{0} = R_{2}$, the winning conditions for Duplicator are \ref{main_1} and \ref{main_2}, as described in Definition~\ref{EHR_first}.
\end{defn}

\begin{lemma}\label{EHR_implies_EHR_alt}
For any two positive integers $r$ and $s$, and two given rooted trees $T_{1}$ and $T_{2}$, if Duplicator wins $\ehr\left[T_{1}, T_{2}, s+1; r\right]$, then she also wins $\EHR\left[T_{1}, T_{2}, s, r\right]$.

\end{lemma}

\begin{proof}
Suppose Spoiler makes $t$ many switches during the game $\EHR\left[T_{1}, T_{2}, s, r\right]$, where $t \leq s$. Without loss of generality, let us assume that Spoiler starts playing on $T_{1}$. Let the first switch happen after the $i_{1}$-th round, the second switch after the $(i_{1}+i_{2})$-th round, and so on, where $t, i_{1}, \ldots, i_{t}$ are entirely dependent on Spoiler's decision and are unknown to Duplicator a priori. Clearly, for each $1 \leq j \leq t$, we have $i_{j} \leq r$.

Having a winning strategy $W$ for Duplicator for the game $\ehr\left[T_{1}, T_{2}, s+1; r\right]$ means that, whatever sequence of $(s+1) r$ moves Spoiler plays according to the rules of the game, Duplicator has a sequence of $(s+1) r$ responses such that she can maintain all the winning conditions. Without loss of generality, suppose Spoiler starts playing $\ehr\left[T_{1}, T_{2}, s+1; r\right]$ on $T_{1}$, and let $x_{1}, \ldots, x_{r}$ denote his moves in $T_{1}$ in the first $r$ rounds, and let $y_{1}, \ldots, y_{r}$ denote the corresponding responses of Duplicator in $T_{2}$ according to $W$; let $y_{r+1}, \ldots, y_{2r}$ denote the moves made by Spoiler in $T_{2}$ in rounds $r+1, \ldots, 2r$, and let $x_{r+1}, \ldots, x_{2r}$ denote the corresponding responses of Duplicator in $T_{1}$ according to $W$, and so on.

Now, we construct the winning strategy for Duplicator for the game $\EHR\left[T_{1}, T_{2}, s, r\right]$. As mentioned above, let Spoiler start playing on tree $T_{1}$. Let $a_{1}, \ldots, a_{i_{1}}$ denote the moves made by Spoiler in rounds $1, \ldots, i_{1}$ in tree $T_{1}$. Setting $x_{j} = a_{j}$ for all $1 \leq j \leq i_{1}$, Duplicator replies with moves $b_{j} = y_{j}$ for all $1 \leq j \leq i_{1}$ where the $y_{j}$'s, as mentioned above, are chosen according to strategy $W$. Note that Duplicator can do this because $i_{1} \leq r$. Next, for $i_{1}+1 \leq j \leq i_{1} + i_{2}$, Spoiler plays on $T_{2}$ and selects vertices $b_{j}, i_{1}+1 \leq j \leq i_{1}+i_{2}$. Duplicator, setting $y_{r+j} = b_{i_{1}+j}$ for all $1 \leq j \leq i_{2}$, selects vertices $a_{i_{1}+j} = x_{r+j}$, for $1 \leq j \leq i_{2}$, in $T_{1}$ according to winning strategy $W$. Once again, this is possible since $i_{2} \leq r$. She continues to play like this until all $r$ rounds are done.

Clearly, the winning conditions listed in Definition~\ref{EHR_first} now only need to hold for a subset $\left\{(x_{j}, y_{j}), 1 \leq j \leq i_{1}\right\} \cup \left\{(x_{j}, y_{j}), r+1 \leq j \leq r+i_{2}\right\} \cup \cdots \left\{(x_{j}, y_{j}), (t-1)r+1 \leq j \leq (t-1)r + i_{t}\right\} \cup \left\{(x_{j}, y_{j}), tr+1 \leq j \leq tr + \left(r - \sum_{\ell=1}^{t} i_{\ell}\right)\right\}$ of the set $\left\{(x_{j}, y_{j}), 1 \leq j \leq r s\right\}$, and we know that the winning conditions already hold for the bigger set because Duplicator chose her responses according to the winning strategy $W$. This shows that she wins $\EHR\left[T_{1}, T_{2}, s, r\right]$.

\end{proof}

We give here yet another version of the Ehrenfeucht games, which comes in handy in the description of winning strategies for Duplicator in Subsection~\ref{start on T_2}. 
\begin{defn}\label{weaker_EHR}
Given rooted trees $T_{1}$ with root $R_{1}$ and $T_{2}$ with root $R_{2}$, and positive integers $s$, $i_{1}, i_{2}, \ldots i_{s}$, this version of the Ehrenfeucht game, denoted $\ehrw\left[T_{1}, T_{2}; i_{1}, \ldots i_{s}\right]$, consists of $i_{1} + i_{2} \ldots + i_{s}$ rounds. First, Spoiler chooses any of $T_{1}$ and $T_{2}$, and makes his first $i_{1}$ moves on that tree, while Duplicator makes the corresponding $i_{1}$ moves on the other tree. Spoiler makes the first switch after the $i_{1}$-th round, the second switch after the $(i_{1}+i_{2})$-th round, $\ldots$, and finally, the $(s-1)$-st switch after the $(i_{1}+ \cdots + i_{s-1})$-th round. 

As before, if $x_{i}$ is the vertex selected from $T_{1}$ and $y_{i}$ that from from $T_{2}$ in round $i$, for $1 \leq i \leq \sum_{j=1}^{s}i_{j}$, then, setting $x_{0} = R_{1}$ and $y_{0} = R_{2}$, Duplicator wins the game if \ref{main_1} and \ref{main_2} hold.
\end{defn}
Notice that this is weaker than the game described in Definition~\ref{EHR_first} because here the values $i_{1}, \ldots, i_{s}$ are known to both Spoiler and Duplicator before the start of the game. 

\begin{lemma}\label{EHR_alt_implies_EHR_weak}
For any two positive integers $s$ and $k$ and rooted trees $T_{1}$ and $T_{2}$, if Duplicator wins $\ehr[T_{1}, T_{2}, s; k]$, then she also wins $\ehrw[T_{1}, T_{2}; i_{1}, \ldots i_{s}]$ for any $1 \leq i_{1}, \ldots i_{s} \leq k$.
\end{lemma}

\begin{proof}
The proof of this lemma is very similar to that of Lemma~\ref{EHR_implies_EHR_alt}, and is in fact simpler, and the details are therefore omitted.
\end{proof}

We now describe the way we put the Ehrenfeucht game described in Definition~\ref{EHR_alt} to use in proving that $\kein_{s}$ has a.q.d.\ $s$ for each $s \in \mathbb{N}$. For any positive integer $k$, we construct two rooted trees $T_{1}$ and $T_{2}$ (these trees will obviously depend on $k$ and $s$), such that $T_{1} \models \kein_{s}$ and $T_{2} \models \neg \kein_{s}$, and Duplicator wins $\ehr\left[T_{1}, T_{2}, s; k\right]$. From this, we can draw the following conclusion: using Lemma~\ref{EHR_implies_EHR_alt}, we know that Duplicator wins $\EHR\left[T_{1}, T_{2}, s-1, k\right]$, and hence, from Theorem~\ref{EHR_logic}, we can tell that for every FO sentence $A$ that has q.d.\ at most $k$ and a.q.d.\ at most $s-1$, either $A$ holds for both $T_{1}$ and $T_{2}$, or $\neg A$ holds for both $T_{1}$ and $T_{2}$. But notice that $\kein_{s}$ holds for $T_{1}$ whereas $\neg \kein_{s}$ holds for $T_{2}$. This clearly means that $\kein_{s}$ either fails to have q.d.\ at most $k$, or it fails to have a.q.d.\ at most $s-1$. Since we are able to construct $T_{1}$ and $T_{2}$ and provide a winning strategy for Duplicator for arbitrary $k$, clearly it is the upper bound on the a.q.d.\ of $\kein_{s}$ that fails. This shows that $\kein_{s}$ must have a.q.d.\ at least $s$. Since we have noted before that its a.q.d.\ is at most $s$, it must be precisely equal to $s$. This completes the proof of our main result.

\section{Construction of the trees}\label{sec:construction}
In this section, we describe the construction of the trees $T_{1}$ and $T_{2}$ depending on the given parameters $s$ and $k$. We take into account an additional parameter $m$, whose role becomes clear from the construction. The constructions are described inductively on $s$, starting with the base case of $s = 1$, and arbitrary $k$, which we now describe.

For any positive integer $m \geq k$, we construct the rooted trees $T_{1}^{(1,k,m)}$ and $T_{2}^{(1,k,m)}$ as follows:
\begin{enumerate}
\item In $T_{1}^{(1,k,m)}$, the root $R_{1}$ has $m+1$ children $u_{1}, \ldots u_{m+1}$, and each of them has $m$ childless children of its own.
\item In $T_{2}^{(1,k,m)}$, the root $R_{2}$ has $m+1$ children $v_{1}, \ldots v_{m+1}$. Each of $v_{1}, \ldots v_{m}$ has $m$ childless children of its own; $v_{m+1}$ has no child.
\end{enumerate}
An illustration is given in Figure~\ref{base_2}.
\begin{figure}
\centering
\includegraphics[width = 0.8\textwidth]{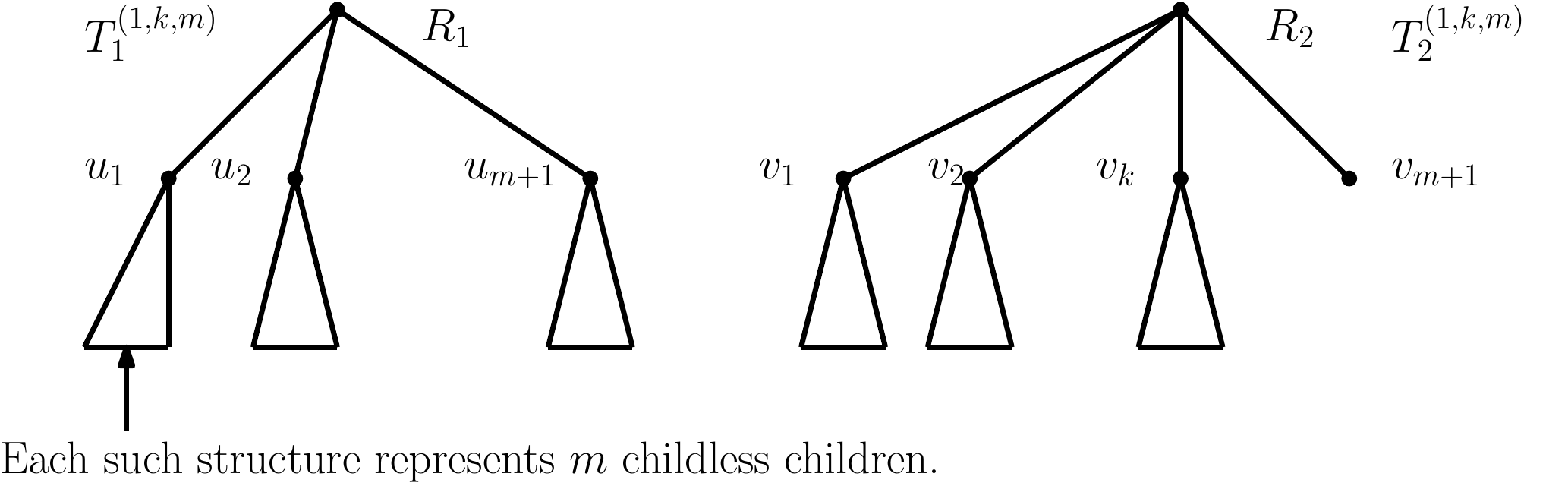}
\caption{$T_{1}^{(1,k,m)}$ and $T_{2}^{(1,k,m)}$}
\label{base_2}
\end{figure} 

Suppose we know how to construct the trees $T_{1}^{(s',k,m)}$ and $T_{2}^{(s',k,m)}$ for all $s' \leq s$ where $s$ is some positive integer, for arbitrary $k$, and for all $m \geq s' k$, such that $T_{1}^{(s',k,m)} \models \kein_{s'}$ and $T_{2}^{(s',k,m)} \models \neg \kein_{s'}$, and Duplicator wins $\ehr\left[T_{1}^{(s',k,m)}, T_{2}^{(s',k,m)}, s'; k\right]$. Now, fixing an arbitrary $k$, we construct $T_{1}^{(s+1,k,m)}$ and $T_{2}^{(s+1,k,m)}$ for all $m \geq (s+1)k$, such that $T_{1}^{(s+1,k,m)} \models \kein_{s+1}$ whereas $T_{2}^{(s+1,k,m)} \models \neg \kein_{s+1}$, and Duplicator wins $\ehr\left[T_{1}^{(s+1,k,m)}, T_{2}^{(s+1,k,m)}, s+1; k\right]$.

Note that we make use of $T_{1}^{(s,k,m)}$ and $T_{2}^{(s,k,m)}$ in constructing $T_{1}^{(s+1,k,m)}$ and $T_{2}^{(s+1,k,m)}$ for $m \geq (s+1)k$, and since this last condition ensures that $m > s k$, hence it makes sense to talk about $T_{1}^{(s,k,m)}$ and $T_{2}^{(s,k,m)}$. The inductive construction is as follows:
\begin{enumerate}
\item In $T_{1}^{(s+1,k,m)}$, the root $R_{1}$ has $m+1$ children $u_{1}, \ldots u_{m+1}$, and from each of them hangs a copy of $T_{2}^{(s,k,m)}$.
\item In $T_{2}^{(s+1,k,m)}$, the root $R_{2}$ has $m+1$ children, $v_{1}, \ldots v_{m+1}$. From each of $v_{1}, \ldots v_{m}$, hangs a copy of $T_{2}^{(s,k,m)}$, and from $v_{m+1}$ hangs a copy of $T_{1}^{(s,k,m)}$.
\end{enumerate}
An illustration is given in Figure \ref{fig:s+1}. 
\begin{figure}[h!]
    \includegraphics[width=0.8\textwidth]{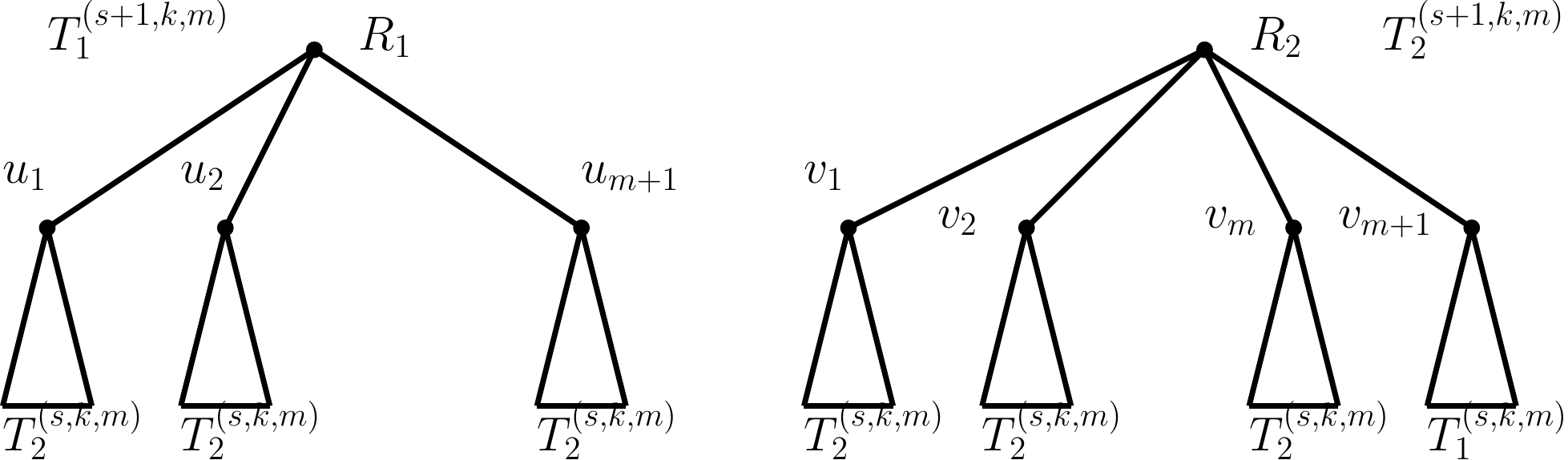}
  \caption{Trees $T_{1}^{(s+1,k,m)}$ and $T_{2}^{(s+1,k,m)}$}
\label{fig:s+1}
\end{figure}

We make sure here that indeed we have $T_{1}^{(s+1,k,m)} \models \kein_{s+1}$ and $T_{2}^{(s+1,k,m)} \models \neg \kein_{s+1}$. Recall from \eqref{recur_prop} that a rooted tree $T$ satisfies $\kein_{s+1}$ if for every child $x$ of the root $R$, the property $\neg P_{s}(x)$ holds. By induction hypothesis, $T_{2}^{(s,k,m)} \models \neg \kein_{s}$. Since for each child $u_{i}$ of the root in $T_{1}^{(s+1,k,m)}$, we have 
$$T_{1}^{(s+1,k,m)}(u_{i}) \cong T_{2}^{(s,k,m)}$$
(where the notation $\cong$ indicates a tree ismorphism map which maps the root of one tree to the root of the other), hence $\neg P_{s}(u_{i})$ holds for every $1 \leq i \leq m+1$. On the other hand, also by induction hypothesis, $T_{1}^{(s,k,m)} \models \kein_{s}$. In $T_{2}^{(s+1,k,m)}$, the root has one child $v_{m+1}$ such that 
$$T_{2}^{(s+1,k,m)}(v_{m+1}) \cong T_{1}^{(s,k,m)},$$
and hence $P_{s}(v_{m+1})$ holds. This tells us that $\neg \kein_{s+1}$ must hold for $T_{2}^{(s+1,k,m)}$. 

\section{Winning strategy for Duplicator for arbitrary $s$}\label{sec:winning}

We first describe Duplicator's winning strategy for $s = 1$. For $s = 1$, no alternation is allowed in $\ehr\left[T_{1}, T_{2}, s; k\right]$, i.e.\ either Spoiler plays the entire game on $T_{1}$ while Duplicator answers on $T_{2}$, or Spoiler plays the entire game on $T_{2}$ and Duplicator answers on $T_{1}$. Also, there are a total of $k$ many rounds now. We now describe the winning strategy for Duplicator in either scenario.

We introduce here a terminology that we shall use in our exposition of the winning strategy for Duplicator henceforth.  
\begin{defn}\label{free so far}
Suppose $i$ rounds of the game $\ehr\left[T_{1}, T_{2}, s; k\right]$ have been played. For any vertex $u \in T_{1}$, we call $u$ \emph{free up to round $i$} if no $x_{j}$, $1 \leq j \leq i$, has been selected from $T_{1}(u)$. Similarly, for any vertex $v \in T_{2}$, we call $v$ free up to round $i$ if no $y_{j}$, $1 \leq j \leq i$, has been selected from $T_{2}(v)$.
\end{defn}

First, consider the case where Spoiler plays on tree $T_{2}^{(1,k,m)}$. Throughout the game, Duplicator maintains the following conditions, and we show that she is able to maintain them using an inductive argument. Suppose $i$ rounds of the game have been played. Then, for $1 \leq j \leq i$, 
\begin{enumerate}
\item if Spoiler selects $y_{j} = v_{\ell}$ for some $1 \leq \ell \leq m+1$, then Duplicator selects $x_{j} = u_{\ell}$;
\item if Spoiler selects $y_{j}$ to be a child of the vertex $v_{\ell}$, for some $1 \leq \ell \leq m$, then Duplicator selects $x_{j}$ to be a child of $u_{\ell}$, making sure that \ref{main_2} is maintained (note that these two conditions together imply that $x_{j} = R_{1}$ if and only if $y_{j} = R_{2}$).
\end{enumerate}
Suppose Duplicator has been able to maintain these conditions up to and including round $i$. Now suppose Spoiler selects $y_{i+1}$ to be equal to some $v_{\ell}$ for $1 \leq \ell \leq m+1$. Then Duplicator sets $x_{i+1} = u_{\ell}$. If Spoiler selects $y_{i+1}$ to be a child of $v_{\ell}$ for some $1 \leq \ell \leq m$, and this child was not chosen before, then Duplicator selects $x_{i+1}$ to be a child of $u_{\ell}$ that has not been chosen before, and notice that such a child she can always find because each $u_{\ell}$ for $1 \leq \ell \leq m$ has $m$ children and $m \geq k$, and there are only $k$ many rounds. If Spoiler selects $y_{i+1}$ to be a child of $v_{\ell}$ that was already chosen before, say in the $j$-th round for some $j \leq i$, then Duplicator, to maintain \ref{main_2}, simply sets $x_{i+1} = x_{j}$. It is straightforward to see that these choices do guarantee the satisfaction of both \ref{main_1} and \ref{main_2} at the end of the game.

Now suppose Spoiler plays on tree $T_{1}^{(1,k,m)}$. Once again, Duplicator maintains the following conditions throughout the game, which we prove via inductive arguments that she is able to. Suppose $i$ rounds of the game have been played. Then for all $j, j' \in [i]$,
\begin{enumerate}
\item for $j \neq j'$, $x_{j} = x_{j'}$ if and only if $y_{j} = y_{j'}$ (this implies that $x_{j} = R_{1}$ if and only if $y_{j} = R_{2}$);
\item $x_{j} \in \left\{u_{1}, \ldots, u_{m+1}\right\}$ if and only if $y_{j} \in \left\{v_{1}, \ldots, v_{m}\right\}$;
\item $x_{j}$ is a child of $u_{\ell}$ for some $1 \leq \ell \leq m+1$ if and only if $y_{j}$ is a child of $v_{\ell'}$ for some $1 \leq \ell' \leq m$;
\item for $j \neq j'$, $\pi(x_{j'}) = x_{j}$ if and only if $\pi(y_{j'}) = y_{j}$;
\item $x_{j} = u_{\ell}$ for some $1 \leq \ell \leq m+1$ such that $u_{\ell}$ has been free up to round $j-1$ if and only if $y_{j} = v_{\ell'}$ for some $1 \leq \ell' \leq m$ such that $v_{\ell'}$ has been free up to round $j-1$;
\item for $j \neq j'$, $x_{j}$ and $x_{j'}$ are siblings, i.e.\ they share a common parent, if and only if $y_{j}$ and $y_{j'}$ are siblings.
\end{enumerate}
Suppose Duplicator has been able to maintain all of these conditions up to and including round $i$, for some $i \leq k-1$. Of course, if Spoiler selects $x_{i+1} = x_{j}$ for some $j \leq i$, then Duplicator sets $y_{i+1} = y_{j}$. So, for the rest of the inductive argument, assume that $x_{i+1}$ is distinct from all previously chosen vertices.

Suppose Spoiler selects $x_{i+1}$ to be some $u_{\ell}$ for $1 \leq \ell \leq m+1$ such that $u_{\ell}$ has been free up to round $i$. Note that $m \geq k$ and there are $k$ rounds in total, hence there has to be at least one vertex $v_{\ell'}$ for some $1 \leq \ell' \leq m$ free up to round $i$. Hence Duplicator finds such a $v_{\ell'}$ and sets $y_{i+1} = v_{\ell'}$. Suppose Spoiler selects $x_{i+1}$ to be $u_{\ell}$ such that there exists $j \leq i$ with $x_{j}$ a child of $u_{\ell}$; by induction hypothesis, we know that $y_{j}$ is a child of $v_{\ell'}$ for some $1 \leq \ell' \leq m$, and we set $y_{i+1} = v_{\ell'}$. 

If Spoiler selects $x_{i+1}$ to be a child of some vertex $u_{\ell}$ for $1 \leq \ell \leq m+1$ such that $u_{\ell}$ has been free up to round $i$, then Duplicator once again finds $v_{\ell'}$ for some $1 \leq \ell' \leq m$ that has been free up to round $i$ (possible for the same reason as argued above), and then selects $y_{i+1}$ to be any child of $v_{\ell'}$. Suppose Spoiler selects $x_{i+1}$ to be a child of $u_{\ell}$, for some $1 \leq \ell \leq m+1$, such that for some $j \leq i$, we have $x_{j} = u_{\ell}$. Then by induction hypothesis, we know that $y_{j}$ must equal $v_{\ell'}$ for some $1 \leq \ell' \leq m$. Duplicator selects a child of $v_{\ell'}$, distinct from all previously chosen vertices, as $y_{i+1}$. As there are $m$ children of $v_{\ell'}$ and a total of $k$ rounds with $m \geq k$, hence she is able to make such a selection. Suppose Spoiler selects $x_{i+1}$ to be a child of $u_{\ell}$ for some $1 \leq \ell \leq m+1$ such that for some $j \leq i$, the vertex $x_{j}$ is also a child of $u_{\ell}$. Again by induction hypothesis we know that $y_{j}$ is a child of $v_{\ell'}$ for some $1 \leq \ell' \leq m$, and Duplicator sets $y_{i+1}$ to be a child of $v_{\ell'}$ that is different from $y_{j}$ as well as any other previously chosen vertex. Note that, since each $u_{\ell}$ and each $v_{\ell'}$ has $m$ children and $m \geq k$, where $k$ is the total number of rounds, there is always a choice of $y_{i+1}$ as a child of $v_{\ell'}$ which does not coincide with any previously chosen vertex. 

This exhausts all possible moves by Spoiler and we have shown that Duplicator is able to respond in each case such that all the conditions mentioned above are maintained. 

We now come to the winning strategy for Duplicator on $\ehr\left[T_{1}^{(s,k,m)}, T_{2}^{(s,k,m)}, s; k\right]$ for arbitrary $k$ and $m \geq s k$, and this is described as an inductive strategy where the induction happens on $s$. Thus, we assume that we already have a winning strategy for Duplicator for the game $\ehr\left[T_{1}^{(s,k,m)}, T_{2}^{(s,k,m)}, s; k\right]$, and then devise a strategy for her for the game $\ehr\left[T_{1}^{(s+1,k,m)}, T_{2}^{(s+1,k,m)}, s+1; k\right]$ for the same $k$ and $m \geq (s+1)k$. The analysis is split into two parts, according to which tree Spoiler starts playing on, and these are described separately, in detail, in Subsections \ref{start on T_2} and \ref{start on T_1}.

\subsection{Spoiler starts playing on $T_{2}^{(s+1,k,m)}$}\label{start on T_2}
Suppose Spoiler starts playing the game $\ehr\left[T_{1}^{(s+1,k,m)}, T_{2}^{(s+1,k,m)}, s+1; k\right]$ on $T_{2}^{(s+1,k,m)}$. Recall, from the inductive construction of the trees in Section~\ref{sec:construction}, that $T_{1}^{(s,k,m)}$ and $T_{2}^{(s,k,m)}$ are as illustrated in Figure~\ref{fig:s}.
\begin{figure}[h!]
  \centering
    \includegraphics[width=0.8\textwidth]{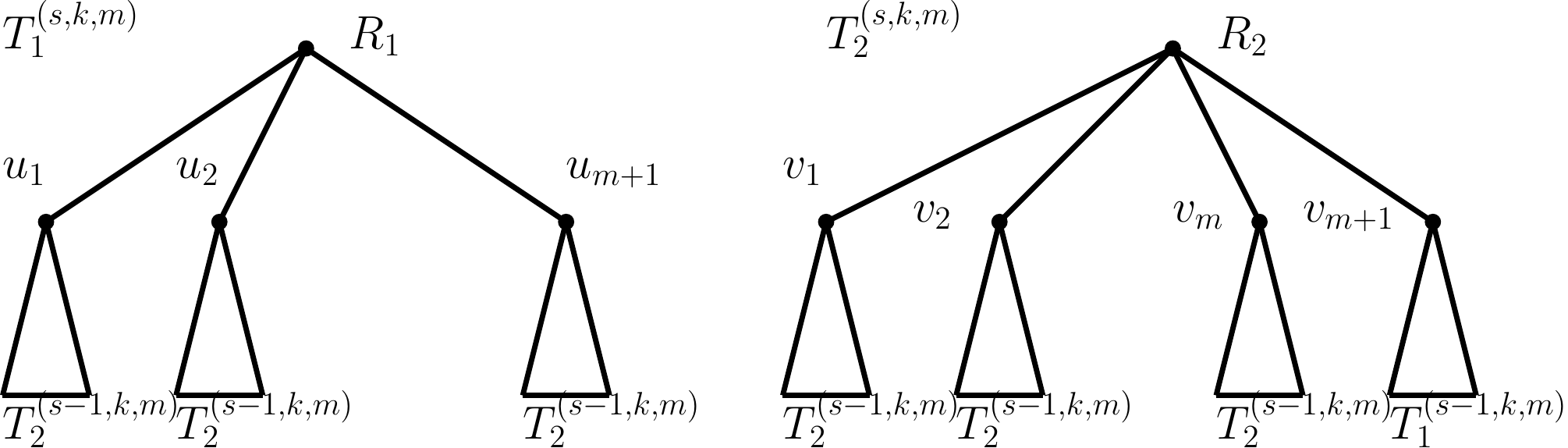}
  \caption{Trees $T_{1}^{(s,k,m)}$ and $T_{2}^{(s,k,m)}$}
  \label{fig:s}
\end{figure}
 
\begin{ind}\label{induction hyp 1}
For every $1 \leq t \leq m+1$ and $1 \leq t' \leq m$, fix a tree isomorphism $\varphi^{(s)}_{t, t'}: T_{1}^{(s,k,m)}(u_{t}) \rightarrow T_{2}^{(s,k,m)}(v_{t'})$ (this is possible as both are copies of $T_{2}^{(s-1,k,m)}$), such that $\varphi^{(s)}_{t,t'}(u_{t}) = v_{t'}$. Fix $\ell$ pairs of vertices $(x_{1}, y_{1}), \ldots (x_{\ell}, y_{\ell})$ in $T_{1}^{(s,k,m)} \times T_{2}^{(s,k,m)}$ such that $\ell \leq k$ and they satisfy the following conditions:
\begin{enumerate}[label=(IH\arabic{*})]
\item \label{ind hyp 1} for each $1 \leq i \leq \ell$, we have $x_{i} \in \bigcup_{t=1}^{m+1} T_{1}^{(s,k,m)}(u_{t})$ and $y_{i} \in \bigcup_{t'=1}^{m} T_{2}^{(s,k,m)}(v_{t'})$.
\item \label{ind hyp 2} For $1 \leq i \neq j \leq \ell$, if $x_{i}, x_{j} \in T_{1}^{(s,k,m)}(u_{t})$ for some $1 \leq t \leq m+1$, then we can find some $1 \leq t' \leq m$ such that $y_{i}, y_{j} \in T_{2}^{(s,k,m)}(v_{t'})$. The converse also holds, i.e.\ if for $1 \leq i \neq j \leq \ell$, if $y_{i}, y_{j} \in T_{2}^{(s,k,m)}(v_{t'})$ for some $1 \leq t' \leq m$, then we can find some $1 \leq t \leq m+1$ such that $x_{i}, x_{j} \in T_{1}^{(s,k,m)}(u_{t})$.
\item \label{ind hyp 3} For every $1 \leq i \leq \ell$, if $x_{i} \in T_{1}^{(s,k,m)}(u_{t})$ and $y_{i} \in T_{2}^{(s,k,m)}(v_{t'})$ for some $1 \leq t \leq m+1$ and $1 \leq t' \leq m$, then $y_{i} = \varphi^{(s)}_{t,t'}(x_{i})$.
\end{enumerate}
As in Definitions~\ref{EHR_first} and \ref{EHR_alt}, we set $x_{0} = R_{1}$, the root of $T_{1}^{(s,k,m)}$, and $y_{0} = R_{2}$, the root of $T_{2}^{(s,k,m)}$. Then Duplicator will be able to win $\ehr\left[T_{1}^{(s,k,m)}, T_{2}^{(s,k,m)}, s; k\right]$, where Spoiler starts playing on $T_{2}^{(s,k,m)}$, with the pairs $(x_{0}, y_{0}), \ldots (x_{\ell}, y_{\ell})$ considered as designated pairs. 
\end{ind} 

\begin{rem}
It is helpful to imagine designated pairs of vertices as pairs chosen \emph{before} the game has even begun. For example, in Induction hypothesis~\ref{induction hyp 1}, one could visualize the designated pairs $(x_{0}, y_{0}), \ldots (x_{\ell}, y_{\ell})$ as pairs already selected / provided to the players from $\ell+1$ many rounds that occurred \emph{before the actual game}, and then the game $\ehr\left[T_{1}^{(s,k,m)}, T_{2}^{(s,k,m)}, s; k\right]$ of $s k$ many rounds begins, which Duplicator is able to win, taking into account these $(\ell+1)$ many pairs as well, if the conditions given in Induction hypothesis~\ref{induction hyp 1} hold. In other words, suppose in $\ehr\left[T_{1}^{(s,k,m)}, T_{2}^{(s,k,m)}, s; k\right]$, the pair selected in round $i$ from $T_{1}^{(s,k,m)} \times T_{2}^{(s,k,m)}$ is denoted by $(x_{i+\ell}, y_{i+\ell})$, for all $1 \leq i \leq s k$. Then Duplicator has to maintain \emph{all} of the following conditions to win $\ehr\left[T_{1}^{(s,k,m)}, T_{2}^{(s,k,m)}, s; k\right]$ with $(x_{0}, y_{0}), \ldots (x_{\ell}, y_{\ell})$ as designated pairs:
\begin{enumerate}
\item $x_{i} = \pi(x_{j}) \Leftrightarrow y_{i} = \pi(y_{j})$ for all $0 \leq i \neq j \leq s k + \ell$;
\item $x_{i} = x_{j} \Leftrightarrow y_{i} = y_{j}$ for all $0 \leq i \neq j \leq s k + \ell$.
\end{enumerate}
\end{rem}

\par Armed with Induction hypothesis~\ref{induction hyp 1}, we now prove that the corresponding claim holds on $T_{1}^{(s+1,k,m)}$ and $T_{2}^{(s+1,k,m)}$ where $m \geq (s+1)k$. First, we need to fix $\ell \leq k$ designated pairs that satisfy conditions analogous to \eqref{ind hyp 1} through \eqref{ind hyp 3}. So, referring to Figure \ref{fig:s+1}, we fix, for every $1 \leq t \leq m+1$ and every $1 \leq t' \leq m$, an isomorphism $\varphi_{t,t'}^{(s+1)}: T_{1}^{(s+1,k,m)}(u_{t}) \rightarrow T_{2}^{(s+1,k,m)}(v_{t'})$ (again possible because both $T_{1}^{(s+1,k,m)}$ and $T_{2}^{(s+1,k,m)}(v_{t'})$ are both copies of $T_{2}^{(s,k,m)}$); next, we fix an arbitrary collection of pairs $(x_{1}, y_{1}), \ldots (x_{\ell}, y_{\ell})$ for any non-negative integer $\ell \leq k$ (if $\ell = 0$, then we choose only the roots $R_{1}$ and $R_{2}$ as $x_{0}$ and $y_{0}$ respectively), such that the following hold: 
\begin{enumerate}[label=(C\arabic{*})]
\item \label{cond 1 designated} for each $1 \leq i \leq \ell$, we have $x_{i} \in \bigcup_{t=1}^{m+1} T_{1}^{(s+1,k,m)}(u_{t})$ and $y_{i} \in \bigcup_{t'=1}^{m} T_{2}^{(s+1,k,m)}(v_{t'})$.
\item \label{cond 2 designated} For $1 \leq i \neq j \leq \ell$, if $x_{i}, x_{j} \in T_{1}^{(s+1,k,m)}(u_{t})$ for some $1 \leq t \leq m+1$, then we can find some $1 \leq t' \leq m$ such that $y_{i}, y_{j} \in T_{2}^{(s+1,k,m)}(v_{t'})$. The converse also holds, i.e.\ if for $1 \leq i \neq j \leq \ell$, if $y_{i}, y_{j} \in T_{2}^{(s+1,k,m)}(v_{t'})$ for some $1 \leq t' \leq m$, then we can find some $1 \leq t \leq m+1$ such that $x_{i}, x_{j} \in T_{1}^{(s+1,k,m)}(u_{t})$.
\item \label{cond 3 designated} For every $1 \leq i \leq \ell$, if $x_{i} \in T_{1}^{(s+1,k,m)}(u_{t})$ and $y_{i} \in T_{2}^{(s+1,k,m)}(v_{t'})$, then $y_{i} = \varphi^{(s+1)}_{t,t'}(x_{i})$.
\end{enumerate}

We now provide a strategy for Duplicator to win $\ehr\left[T_{1}^{(s+1,k,m)}, T_{2}^{(s+1,k,m)}, s+1; k\right]$ with $(x_{0}, y_{0}), (x_{1}, y_{1}) \ldots (x_{\ell}, y_{\ell})$ as designated vertices. Recall again that $x_{0} = R_{1}$, the root of $T_{1}^{(s+1,k,m)}$, and $y_{0} = R_{2}$, the root of $T_{2}^{(s+1,k,m)}$.

\begin{assume}\label{assume 1}
Since $\ell \leq k$ and $m \geq (s+1)k$ with $s \geq 2$, there exists at least one $1 \leq i_{0} \leq m+1$ such that no $x_{j}, 1 \leq j \leq \ell$, has been chosen from $T_{1}^{(s+1,k,m)}(u_{i_{0}})$. Because for each $1 \leq i \leq m+1$, the subtree $T_{1}^{(s+1,k,m)}(u_{i})$ is a copy of $T_{2}^{(s,k,m)}$, we can therefore assume without loss of generality that $i_{0} = m+1$.
\end{assume}
Assumption \ref{assume 1} is made throughout Subsection~\ref{start on T_2}.

\par Let us now take a closer look at the trees $T_{1}^{(s+1,k,m)}$ and $T_{2}^{(s+1,k,m)}$, especially in more detail the subtrees $T_{1}^{(s+1,k,m)}(u_{m+1})$ and $T_{2}^{(s+1,k,m)}(v_{m+1})$. These are illustrated in Figures \eqref{fig:s+1:1:detailed} and \eqref{fig:s+1:2:detailed}.

\begin{figure}[h!]
\includegraphics[width = 0.7\textwidth]{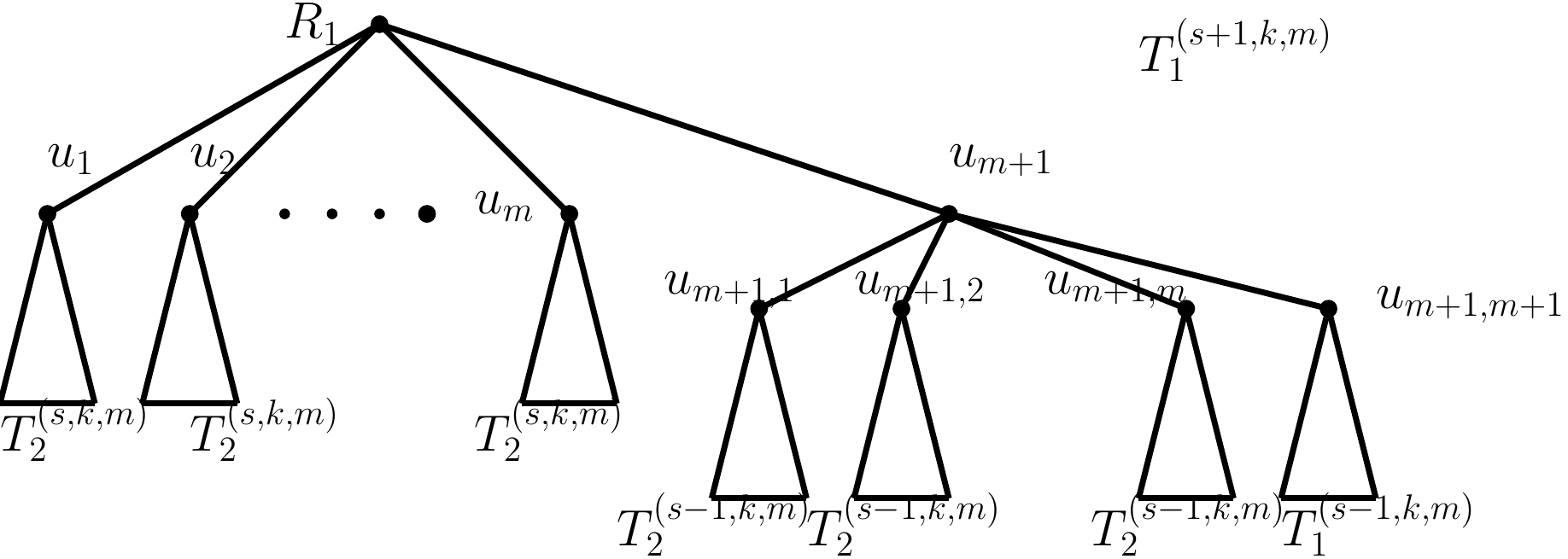}
\caption{$T_{1}^{(s+1,k,m)}$ with detailed view of $T_{1}^{(s+1,k,m)}(u_{m+1})$}
\label{fig:s+1:1:detailed}
\end{figure}

\begin{figure}[h!]
\includegraphics[width = 0.7\textwidth]{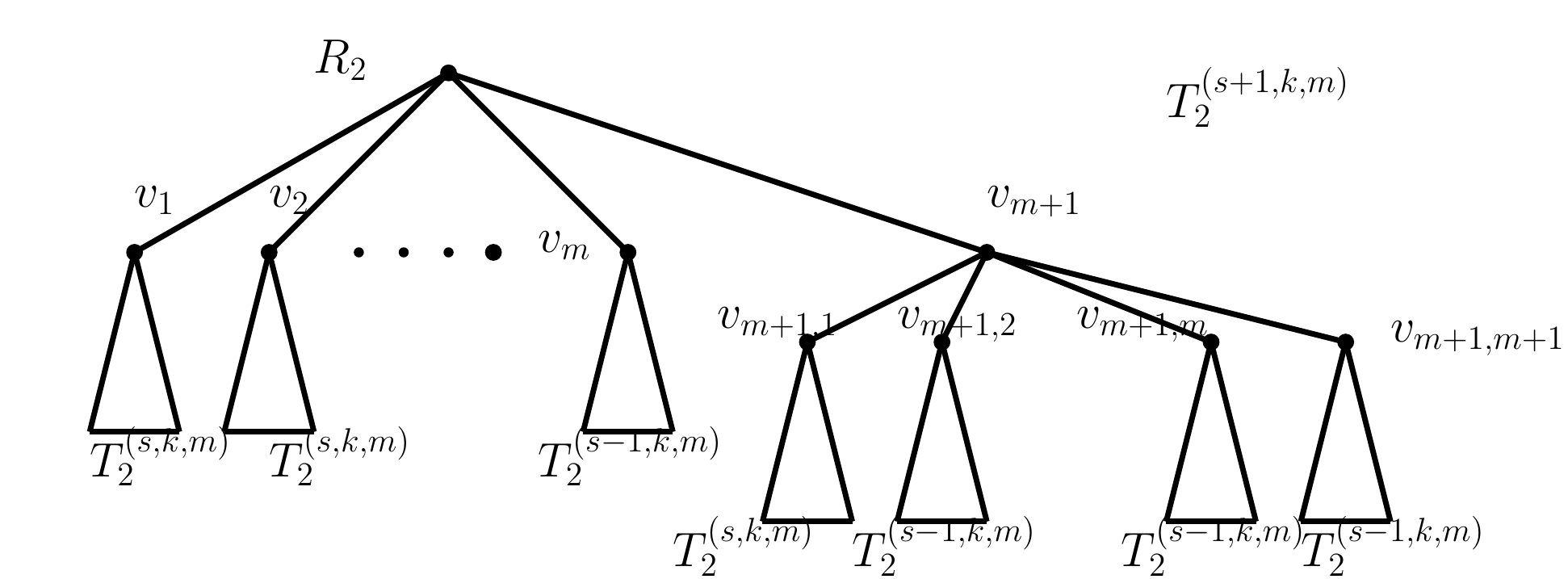}
\caption{$T_{2}^{(s+1,k,m)}$ with detailed view of $T_{2}^{(s+1,k,m)}(v_{m+1})$}
\label{fig:s+1:2:detailed}
\end{figure}

\par Let $(x_{i+\ell}, y_{i+\ell})$ be the pair selected from $T_{1}^{(s+1,k,m)} \times T_{2}^{(s+1,k,m)}$ in round $i$ for $1 \leq i \leq (s+1) k$. As is our convention for Subsection~\ref{start on T_2}, Spoiler plays the first $k$ rounds on $T_{2}^{(s+1,k,m)}$. In the following paragraph, we state some conditions Duplicator maintains on the configuration $\left\{(x_{0}, y_{0}), \ldots, (x_{\ell}, y_{\ell})\right\} \cup \{(x_{i+\ell}, y_{i+\ell}): 1 \leq i \leq k\}$ (i.e.\ on the pairs of vertices resulting from the first $k$ rounds, along with the designated pairs). And we prove that she can indeed maintain these conditions by using an inductive argument (within the first $k$ rounds). 

First, Duplicator fixes any tree isomorphism $\varphi_{t,t'}^{(s)}: T_{1}^{(s+1,k,m)}(u_{m+1,t}) \rightarrow T_{2}^{(s+1,k,m)}(v_{m+1,t'})$, with $\varphi_{t,t'}^{(s)}(u_{m+1,t}) = v_{m+1,t'}$, for all $1 \leq t \leq m$ and $1 \leq t' \leq m+1$. This is possible as each $T_{1}^{(s+1,k,m)}(u_{m+1,t})$, $1 \leq t \leq m$, as well as each $T_{2}^{(s+1,k,m)}(v_{m+1,t'})$, $1 \leq t' \leq m+1$, is a copy of $T_{2}^{(s-1,k,m)}$. Suppose $p \leq k-1$ rounds of the game have been played. The conditions on the configuration $\left\{(x_{0}, y_{0}), \ldots, (x_{\ell}, y_{\ell})\right\} \cup \left\{(x_{\ell+1}, y_{\ell+1}), \ldots (x_{\ell+p}, y_{\ell+p})\right\}$ are as follows: 

\begin{enumerate}[label=(A\arabic{*})]

\item\label{first k rounds 1} $x_{i+\ell} = R_{1} \Leftrightarrow y_{i+\ell} = R_{2}$. If $x_{i+\ell} \in T_{1}^{(s+1,k,m)}(u_{t})$ for some $1 \leq t \leq m$, then $y_{i+\ell} \in T_{2}^{(s+1,k,m)}(v_{t'})$ for some $1 \leq t' \leq m$, and vice versa; moreover, in this case, $y_{i+\ell} = \varphi^{s+1}_{t, t'}(x_{i+\ell})$. 

Note that if $x_{i+\ell} \in T_{1}^{(s+1,k,m)}(u_{m+1})$, then actually $x_{i+\ell} \in \bigcup_{t=1}^{m} \in T_{1}^{(s+1,k,m)}(u_{m+1,t}) \bigcup \{u_{m+1}\}$. If $x_{i+\ell} = u_{m+1}$ then $y_{i+\ell} = v_{m+1}$, and vice versa. We have $x_{i+\ell} \in T_{1}^{(s+1,k,m)}(u_{m+1,t})$ for some $1 \leq t \leq m$ if and only if $y_{i+\ell} \in T_{2}^{(s+1,k,m)}(u_{m+1,t'})$ for some $1 \leq t' \leq m+1$, and moreover, in this case, $y_{i+\ell} = \varphi^{(s)}_{t,t'}(x_{i+\ell})$.

\item\label{first k rounds 2} Suppose we have $x_{i+\ell}, x_{j} \in T_{1}^{(s+1,k,m)}(u_{t})$ for some $1 \leq i \leq p$, some $1 \leq j \leq p+\ell$, and some $1 \leq t \leq m$. Then there exists some $1 \leq t' \leq m$ such that $y_{i+\ell}, y_{j} \in T_{2}^{(s+1,k,m)}(v_{t'})$.

The converse of this statement is true as well. That is, if for some $1 \leq i \leq p$, some $1 \leq j \leq p+\ell$, and some $1 \leq t' \leq m$, we have $y_{i+\ell}, y_{j} \in T_{2}^{(s+1,k,m)}(v_{t'})$, then there exists some $1 \leq t' \leq m$ such that $x_{i+\ell}, x_{j} \in T_{1}^{(s+1,k,m)}(u_{t})$.

\item\label{first k rounds 3} Suppose for $1 \leq i \neq j \leq p$ we have $x_{i+\ell}, x_{j+\ell} \in T_{1}^{(s+1,k,m)}(u_{m+1, t})$ for some $1 \leq t \leq m$. Then there exists $1 \leq t' \leq m+1$ such that $y_{i+\ell}, y_{j+\ell} \in  T_{2}^{(s+1,k,m)}(v_{m+1,t'})$. The converse of this statement is also true, i.e.\ if for $1 \leq i \neq j \leq p$, we have $y_{i+\ell}, y_{j+\ell} \in T_{2}^{(s+1,k,m)}(v_{m+1, t'})$ for some $1 \leq t' \leq m+1$, then there exists $1 \leq t \leq m$ such that $x_{i+\ell}, x_{j+\ell} \in T_{1}^{(s+1,k,m)}(u_{m+1,t})$.
\end{enumerate}

Suppose Duplicator has been able to maintain all these conditions up to and including round $p$. Now, Spoiler chooses $y_{\ell+p+1}$ from $T_{2}^{(s+1,k,m)}$. Here is how Duplicator replies:
\begin{enumerate}
\item If $y_{p+\ell+1} = R_{2}$, the root of $T_{2}^{(s+1,k,m)}$, then Duplicator chooses $x_{p+\ell+1} = R_{1}$, the root of $T_{1}^{(s+1,k,m)}$.
\item Suppose $y_{p+\ell+1} \in \bigcup_{t'=1}^{m} T_{2}^{(s+1,k,m)}(v_{t'})$. Then there are a few possible cases, as follows:
\begin{enumerate}
\item \label{needed later 1} Suppose there exists $1 \leq j \leq \ell+p$ such that $y_{j}$ and $y_{p+\ell+1}$ both belong to the same $T_{2}^{(s+1,k,m)}(v_{t'})$ for some $1 \leq t' \leq m$. Then Duplicator finds $t$ such that $x_{j} \in T_{1}^{(s+1,k,m)}(u_{t})$. Note that if $1 \leq j \leq \ell$, then $1 \leq t \leq m$ because of Assumption~\ref{assume 1}), and if $\ell+1 \leq j \leq p+\ell$, then $1 \leq t \leq m$ because of induction hypothesis \ref{first k rounds 1}. Then Duplicator sets $x_{p+\ell+1} = \left\{\varphi^{(s+1)}_{t,t'}\right\}^{-1}(y_{p+\ell+1})$.

\item \label{needed later 2} Suppose $y_{p+\ell+1} \in T_{2}^{(s+1,k,m)}(v_{t'})$ for some $1 \leq t' \leq m$, such that $v_{t'}$ is free up to round $p$, i.e.\ $y_{j} \notin T_{2}^{(s+1,k,m)}(v_{t'})$ for all $1 \leq j \leq \ell+p$ (recall Definition~\ref{free so far}). Then Duplicator finds a $1 \leq t \leq m$ such that $u_{t}$ is free up to round $p$. She can always find such a $t$ as $m \geq (s+1)k > 2k > p+\ell$. She then sets $x_{p+\ell+1} = \left\{\varphi^{(s+1)}_{t,t'}\right\}^{-1}(y_{p+\ell+1})$. 

\end{enumerate}

\item Now suppose $y_{p+\ell+1} \in T_{2}^{(s+1,k,m)}(v_{m+1})$. Again, there are a few possible cases:
\begin{enumerate}
\item If $y_{p+\ell+1} = v_{m+1}$, then Duplicator sets $x_{p+\ell+1} = u_{m+1}$.
\item If $y_{p+\ell+1} \in T_{2}^{(s+1,k,m)}(v_{m+1,t'})$ for some $1 \leq t' \leq m+1$ such that there exists some $1 \leq i \leq p$ with $y_{i+\ell} \in T_{2}^{(s+1,k,m)}(v_{m+1,t'})$, then Duplicator finds the $t$ (where $1 \leq t \leq m$ by induction hypothesis \ref{first k rounds 1}) such that $x_{i+\ell} \in T_{1}^{(s+1,k,m)}(u_{m+1,t})$. Then she sets $x_{p+\ell+1} = \left\{\varphi^{(s)}_{t,t'}\right\}^{-1}(y_{p+\ell+1})$.
\item If $y_{p+\ell+1} \in T_{2}^{(s+1,k,m)}(v_{m+1,t'})$ for some $1 \leq t' \leq m+1$ such that $y_{i+\ell} \notin T_{2}^{(s+1,k,m)}(v_{m+1,t'})$ for all $1 \leq i \leq p$, then Duplicator finds a $1 \leq t \leq m$ such that $u_{m+1,t}$ is free up to round $p$ (again, possible since $m \geq (s+1)k > 2k > p+\ell$) and sets $x_{p+\ell+1} = \left\{\varphi^{(s)}_{t,t'}\right\}^{-1}(y_{p+\ell+1})$.

\end{enumerate}
\end{enumerate}

Let us focus now on $T_{2}^{(s+1,k,m)}(v_{m+1})$ and $T_{1}^{(s+1,k,m)}(u_{m+1})$ only. Suppose the only pairs selected up to round $k$ that are in $T_{1}^{(s+1,k,m)}(u_{m+1}) \times T_{2}^{(s+1,k,m)}(v_{m+1})$ are $\left(x_{i_{1}+\ell}, y_{i_{1}+\ell}\right), \ldots \left(x_{i_{r}+\ell}, y_{i_{r}+\ell}\right)$. Now, $T_{1}^{(s+1,k,m)}(u_{m+1})$ is a copy of $T_{2}^{(s,k,m)}$ and $T_{2}^{(s+1,k,m)}(v_{m+1})$ is a copy of $T_{1}^{(s,k,m)}$. From \ref{first k rounds 1} and \ref{first k rounds 3}, we can see that the pairs $(y_{i_{1}+\ell}, x_{i_{1}+\ell}), \ldots (y_{i_{r}+\ell}, x_{i_{r}+\ell})$ satisfy Conditions \ref{ind hyp 1} through \ref{ind hyp 3}. 
\begin{conc}\label{conc 1}
So Duplicator, by Induction hypothesis \ref{induction hyp 1}, will win $\ehr[T_{2}^{(s+1,k,m)}(v_{m+1}), T_{1}^{(s+1,k,m)}(u_{m+1}), s; k]$ with designated pairs $(y_{i_{1}+\ell}, x_{i_{1}+\ell}), \ldots (y_{i_{r}+\ell}, x_{i_{r}+\ell})$ (notice the deliberate writing of $T_{2}^{(s+1,k,m)}(v_{m+1})$ before $T_{1}^{(s+1,k,m)}(u_{m+1})$).
\end{conc}

\par Now consider the subsequent game of remaining $s k$ rounds. We call this the \emph{second part of the game}. For this part of the game, we really can split the tree $T_{1}^{(s+1,k,m)}$ into 
\begin{equation}\label{split 1}
\{R_{1}\} \cup S_{1} \cup T_{1}^{(s+1,k,m)}(u_{m+1}), \text{ where } S_{1} = \left\{\bigcup_{t=1}^{m} T_{1}^{(s+1,k,m)}(u_{t})\right\},
\end{equation}
and the tree $T_{2}^{(s+1,k,m)}$ into
\begin{equation}\label{split 2}
\{R_{2}\} \cup S_{2} \cup T_{2}^{(s+1,k,m)}(v_{m+1}), \text{ where } S_{2} = \left\{\bigcup_{t'=1}^{m} T_{1}^{(s+1,k,m)}(v_{t'})\right\}.
\end{equation}
Suppose $p$ rounds of the game have been played, with $k < p \leq (s+1)k$. Duplicator maintains the following conditions on the configuration $\left\{(x_{i}, y_{i}): 0 \leq i \leq p+\ell\right\}$:
\begin{enumerate}[label=(B\arabic{*})]
\item \label{b_1} For $0 \leq i \leq p$, we have $x_{i} = R_{1} \Leftrightarrow y_{i} = R_{2}$.
\item \label{b_2} For $0 \leq i \leq p$, we have $x_{i} \in S_{1} \Leftrightarrow y_{i} = S_{2}$. In this case, if $x_{i} \in T_{1}^{(s+1,k,m)}(u_{t})$ for some $1 \leq t \leq m$ and $y_{i} \in T_{2}^{(s+1,k,m)}(v_{t'})$ for some $1 \leq t' \leq m$, then $y_{i} = \varphi^{(s+1)}_{t,t'}(x_{i})$.
\item \label{b_3} For $0 \leq i, i' \leq p$, if $x_{i}, x_{i'} \in T_{1}^{(s+1,k,m)}(u_{t})$ for some $1 \leq t \leq m$, then there exists $1 \leq t' \leq m$ with $y_{i}, y_{i'} \in T_{2}^{(s+1,k,m)}(v_{t'})$.
\item \label{b_4} For $0 \leq j \leq p$, we have $x_{j} \in T_{1}^{(s+1,k,m)}(u_{m+1}) \Leftrightarrow y_{j} \in T_{2}^{(s+1,k,m)}(v_{m+1})$. Suppose $(i-1) k + 1 \leq p \leq ik$, where $2 \leq i \leq s+1$. Suppose the number of pairs of vertices selected from $T_{2}^{(s+1,k,m)}(v_{m+1}) \times T_{1}^{(s+1,k,m)}(u_{m+1})$ between rounds $k+1$ and $2k$ is $j_{1}$, between rounds $2k+1$ and $3k$ is $j_{2}$, $\ldots$, between rounds $(i-2)k +1$ and $(i-1)k$ is $j_{i-2}$, between rounds $(i-1)k+1$ and $p$ is $j'$. Then Duplicator wins $\ehrw\left[T_{2}^{(s+1,k,m)}(v_{m+1}), T_{1}^{(s+1,k,m)}(u_{m+1}); j_{1}, j_{2}, \ldots j_{i-2}, j'+ik-p, \underbrace{k, \ldots, k}_\text{$(s+1-i)$ many}\right]$ with designated pairs $(y_{i_{1}+\ell}, x_{i_{1}+\ell}), \ldots (y_{i_{r}+\ell}, x_{i_{r}+\ell})$, where the first $j_{1} + j_{2} + \ldots j_{i-2} + j'$ rounds have been played out, with Spoiler starting on $T_{1}^{(s+1,k,m)}(u_{m+1})$.

\end{enumerate}

It is straightforward to see that Duplicator is able to maintain all the conditions stated above, via a similar inductive argument as has been detailed in this paper before. She is able to maintain Condition \ref{b_4} because of Conclusion \ref{conc 1} and Lemma \ref{EHR_alt_implies_EHR_weak}. This brings us to the end of the inductive proof. \qed

\subsection{When Spoiler starts playing on $T_{1}^{(s,k,m)}$}\label{start on T_1}
The construction remains the same as before. The induction hypothesis in this case is simpler to state. 
\begin{ind}\label{induction hyp 2}
Duplicator wins $\ehr\left[T_{1}^{(s,k,m)}, T_{2}^{(s,k,m)}, s; k\right]$ when Spoiler starts on $T_{1}^{(s,k,m)}$. 
\end{ind}

We show here the proof by induction. Refer to Figures \ref{fig:s+1:1:detailed} and \ref{fig:s+1:2:detailed}. Again, we chalk out the detailed strategy of Duplicator for the first $k$ rounds, where we do not use Induction hypothesis \ref{induction hyp 2}. For every $1 \leq t \leq m+1$ and $1 \leq t' \leq m$, she fixes any tree isomorphism $\varphi^{(s+1)}_{t,t'}: T_{1}^{(s+1,k,m)}(u_{t}) \rightarrow T_{2}^{(s+1,k,m)}(v_{t'})$ with $\varphi^{(s+1)}_{t,t'}(u_{t}) = v_{t'}$. This is possible since each of $T_{1}^{(s+1,k,m)}(u_{t})$ and $T_{2}^{(s+1,k,m)}(v_{t'})$ is a copy of $T_{2}^{(s,k,m)}$ for all $1 \leq t \leq m+1$ and $1 \leq t' \leq m$. For the first $p$ rounds, for all $1 \leq p \leq k$, she maintains the following conditions on the configuration $\left\{(x_{i}, y_{i}): 1 \leq i \leq p\right\}$:
\begin{enumerate}[label=(A'\arabic{*})]
\item For all $1 \leq i \leq p$, we have $x_{i} = R_{1} \Leftrightarrow y_{i} = R_{2}$.
\item For all $1 \leq i \leq p$, we have $x_{i} \in T_{1}^{(s+1,k,m)}(u_{t})$ for some $1 \leq t \leq m+1$ if and only if $y_{i} \in T_{2}^{(s+1,k,m)}(v_{t'})$ for some $1 \leq t' \leq m$. Furthermore $y_{i} = \varphi^{(s+1)}_{t,t'}(x_{i})$.
\item If $x_{i}, x_{j} \in T_{1}^{(s+1,k,m)}(u_{t})$ for some $1 \leq i \neq j \leq p$ and some $1 \leq t \leq m+1$, then there exists some $1 \leq t' \leq m$ such that $y_{i}, y_{j} \in T_{2}^{(s+1,k,m)}(v_{t'})$. The converse is also true, i.e.\ if $y_{i}, y_{j} \in T_{2}^{(s+1,k,m)}(v_{t'})$ for some $1 \leq i \neq j \leq p$ and some $1 \leq t' \leq m$, then there exists some $1 \leq t \leq m+1$ such that $x_{i}, x_{j} \in T_{1}^{(s+1,k,m)}(u_{t})$.
\end{enumerate} 

That Duplicator is able to maintain these conditions can be shown by induction on $p$ where $1 \leq p \leq k-1$. So, suppose the first $p$ rounds have been played. In the $(p+1)$-st round Spoiler chooses $x_{p+1}$ from $T_{1}^{(s+1,k,m)}$. Duplicator replies as follows:
\begin{enumerate}
\item If $x_{p+1} = R_{1}$, then Duplicator sets $y_{p+1} = R_{2}$.
\item If $x_{p+1} \in T_{1}^{(s+1,k,m)}(u_{t})$ for some $1 \leq t \leq m+1$ such that there exists some $1 \leq j \leq p$ with $x_{j} \in T_{1}^{(s+1,k,m)}(u_{t})$ as well, then Duplicator finds the $1 \leq t' \leq m$ such that $y_{j} \in T_{2}^{(s+1,k,m)}(v_{t'})$ and sets $y_{p+1} = \varphi^{(s+1)}_{t,t'}(x_{p+1})$.
\item If $x_{p+1} \in T_{1}^{(s+1,k,m)}(u_{t})$ for some $1 \leq t \leq m+1$ such that $x_{j} \notin T_{1}^{(s+1,k,m)}(u_{t})$ for all $1 \leq j \leq p$, then Duplicator finds $1 \leq t' \leq m$ such that $v_{t'}$ has been free up to round $p$. She can always find such a $v_{t'}$ since $ m \geq (s+1)k > k > p$. Then she sets $y_{p+1} = \varphi^{(s+1)}_{t,t'}(x_{p+1})$.
\end{enumerate}

\begin{assume}\label{assume 2}
Since $m \geq (s+1)k > k$, hence there will be at least one $1 \leq i_{0} \leq m+1$ such that $u_{i_{0}}$ is free up to round $k$. Since for all $1 \leq t \leq m+1$, $T_{1}^{(s+1,k,m)}(u_{t})$ is a copy of $T_{2}^{(s,k,m)}$, hence we can assume without loss of generality that $i_{0} = m+1$. 
\end{assume}
We make Assumption~\ref{assume 2} throughout the rest of Subsection~\ref{start on T_1}.

Now we come to the second part of the game, where there are $s k$ many rounds, and Spoiler plays the first $k$ rounds of this part on $T_{2}^{(s+1,k,m)}$. Note that $T_{1}^{(s+1,k,m)}(u_{m+1})$ is a copy of $T_{2}^{(s,k,m)}$ whereas $T_{2}^{(s+1,k,m)}(v_{m+1})$ is a copy of $T_{1}^{(s,k,m)}$. 

\begin{conc}\label{conc 2}
By induction hypothesis, Duplicator wins $\ehr\left[T_{2}^{(s+1,k,m)}(v_{m+1}), T_{1}^{(s+1,k,m)}(u_{m+1}), s; k\right]$ when Spoiler starts playing on $T_{2}^{(s+1,k,m)}(v_{m+1})$.
\end{conc}

We can again split up the two trees as follows:
\begin{equation}
T_{1}^{(s+1,k,m)} = \{R_{1}\} \cup S_{1} \cup T_{1}^{(s+1,k,m)}(u_{m+1}), \text{ where } S_{1} = \bigcup_{t=1}^{m} T_{1}^{(s+1,k,m)}(u_{t}),
\end{equation}
and
\begin{equation}
T_{2}^{(s+1,k,m)} = \{R_{2}\} \cup S_{2} \cup T_{2}^{(s+1,k,m)}(v_{m+1}), \text{ where } S_{2} = \bigcup_{t'=1}^{m} T_{2}^{(s+1,k,m)}(v_{t'}).
\end{equation}
In the second part of the game, for every $k \leq p \leq (s+1)k$, Duplicator maintains the following conditions on the configuration $\left\{(x_{i}, y_{i}): 0 \leq i \leq p\right\}$:
\begin{enumerate}[label=(B'\arabic{*})]
\item \label{b'_1} For $0 \leq i \leq p$, we have $x_{i} = R_{1} \Leftrightarrow y_{i} = R_{2}$.
\item \label{b'_2} For $0 \leq i \leq p$, we have $x_{i} \in S_{1} \Leftrightarrow y_{i} \in S_{2}$. If $x_{i} \in T_{1}^{(s+1,k,m)}(u_{t})$ for some $1 \leq t \leq m$ and $y_{i} \in T_{2}^{(s+1,k,m)}(v_{t'})$ for some $1 \leq t' \leq m$, then $y_{i} = \varphi^{(s+1)}_{t,t'}(x_{i})$.
\item \label{b'_3} For $1 \leq i \neq j \leq p$, we have $x_{i}, x_{j} \in T_{1}^{(s+1,k,m)}(u_{t})$ for some $1 \leq t \leq m$ if and only if there exists some $1 \leq t' \leq m$ with $y_{i}, y_{j} \in T_{2}^{(s+1,k,m)}(v_{t'})$.
\item \label{b'_4} For $0 \leq j \leq p$, we have $x_{j} \in T_{1}^{(s+1,k,m)}(u_{m+1}) \Leftrightarrow y_{j} \in T_{2}^{(s+1,k,m)}(v_{m+1})$. Suppose $(i-1) k + 1 \leq p \leq ik$, where $2 \leq i \leq s+1$. Suppose the number of pairs of nodes selected from $T_{2}^{(s+1,k,m)}(v_{m+1}) \times T_{1}^{(s+1,k,m)}(u_{m+1})$ between rounds $k+1$ and $2k$ is $j_{1}$; between rounds $2k+1$ and $3k$ is $j_{2}$; $\ldots$ between rounds $(i-2)k +1$ and $(i-1)k$ is $j_{i-2}$; between rounds $(i-1)k+1$ and $p$ is $j'$. Then Duplicator wins $\ehrw\left[T_{2}^{(s+1,k,m)}(v_{m+1}), T_{1}^{(s+1,k,m)}(u_{m+1}); j_{1}, j_{2}, \ldots j_{i-2}, j'+ik-p, \underbrace{k, \ldots, k}_\text{$(s+1-i)$ many}\right]$, where the first $j_{1} + j_{2} + \ldots j_{i-2} + j'$ rounds have been played out, with Spoiler starting on $T_{2}^{(s+1,k,m)}(u_{m+1})$.
\end{enumerate}

Again, it is straightforward to see that Duplicator is able to maintain all the conditions stated above. She is able to maintain the Condition~\ref{b'_4} because of Conclusion \ref{conc 2} and Lemma~\ref{EHR_alt_implies_EHR_weak}. This brings us to the end of the inductive proof. \qed

\section{Acknowledgements}
The author expresses her sincere gratitude to her doctoral advisor Joel Spencer for suggesting this beautiful combinatorial problem to her. She also extends her thanks to Maksim Zhukovskii for carefully perusing through the proofs in this paper for correctness of content, and for suggesting further scopes of research in this area. This research has been supported partially by the grant NSF DMS-1444084.


\begin{thebibliography}{99}



\bibitem{01}
J.\ Spencer, The Strange Logic of Random Graphs, Springer Publishing Company, Inc. 2010; Series : Algorithms and Combinatorics, Vol. 22; ISBN:3642074995 9783642074998.


\bibitem{02}
N.\ Immerman, Descriptive complexity, Springer Science \& Business Media, 2012; ISBN-13: 978-0387986005.

\bibitem{03}
D.\ Marker, Model theory: an introduction, Springer Science \& Business Media, 2006; ISBN: 978-1-4419-3157-3.

\bibitem{04}
G.\ Erich, K.\ G.\ Phokion, L.\ Libkin, M.\ Marx, J.\ Spencer, M.\ Y.\ Vardi, Y.\ Venema and S.\ Weinstein, Finite Model Theory and its applications, Springer Science \& Business Media, 2007; eBook ISBN: 978-3-540-68804-4.

\bibitem{05}
T.\ Place and M.\ Zeitoun, The tale of the quantifier alternation hierarchy of first-order logic over words, ACM SIGLOG News, Volume 2, Issue 3, Pages 4--17, 2015, ACM.

\bibitem{06}
M.\ Kufleitner and A.\ Lauser, Quantifier alternation in two-variable first-order logic with successor is decidable, arXiv preprint arXiv:1212.6500, 2012.

\bibitem{07}
T.\ Place and M.\ Zeitoun, Going higher in the first-order quantifier alternation hierarchy on words, International Colloquium on Automata, Languages, and Programming, Pages 342--353, 2014, Springer.

\bibitem{08}
M.\ Boja{\'n}czyk and L.\ Segoufin, Tree languages defined in first-order logic with one quantifier alternation, International Colloquium on Automata, Languages, and Programming, Pages 233--245, 2008, Springer.

\bibitem{09}
H.\ J.\ Keisler and W.\ B.\ Lotfallah, First order quantifiers in monadic second order logic, The Journal of Symbolic Logic, Volume 69, Issue 1, Pages 118--136, 2004, Cambridge University Press. 

\bibitem{10}
E.\ Pezzoli, Computational Complexity of Ehrenfeucht-Fraïssé Games on Finite Structures, International Workshop on Computer Science Logic, Pages 159--170, 1998, Springer.

\bibitem{11}
E.\ Pezzoli, On the computational complexity of type-two functionals and logical games on finite structures, Ph.\ D.\ thesis, Stanford University, 1998.

\bibitem{12}
A.\ D.\ Matushkin, and M.\ E.\ Zhukovskii, ME First order sentences about random graphs: small number of alternations, Discrete Applied Mathematics, Volume 236, Pages 329--346, 2018, Elsevier.

\bibitem{13}
O.\ Matz, N.\ Schweikardt and W.\ Thomas, The monadic quantifier alternation hierarchy over grids and graphs, Information and Computation, Volume 179, Issue 2, Pages 356--383, 2002, Academic Press. 

\bibitem{14}
O.\ Matz, Dot-depth, monadic quantifier alternation, and first-order closure over grids and pictures, Theoretical Computer Science, Volume 270, Issue 1-2, Pages 1--70, 2002, Elsevier. 

\bibitem{15}
A.\ Chandra and D.\ Harel, Structure and complexity of relational queries, Journal of Computer and System Sciences, Volume 25, Issue 1, Pages 99--128, 1982, Elsevier.

\bibitem{16}
W.\ Thomas, Classifying regular events in symbolic logic, Journal of Computer and System Sciences, Volume 25, Issue 3, Pages 360--376, 1982, Elsevier. 

\bibitem{17}
O.\ Matz and N.\ Schweikardt, Expressive power of monadic logics on words, trees, pictures, and graphs, Logic and Automata, Volume 2, Pages 531--552, 2008.


\end{thebibliography}
\end{document}